\documentclass[envcountsame]{llncs}[12/12/07]
\usepackage{psfrag}

\usepackage{picins}
\def\abs#1{\left \vert #1 \right \vert}
\def\Frac#1#2{{\displaystyle{\frac{#1}{#2}}}}
\def\[#1\]{\begin{equation}#1\end{equation}}
\def\$#1\${\begin{eqnarray}#1\end{eqnarray}}
\def\phi{\varphi}

\def\mod#1{\;[\hbox{\rm mod}\; #1]}

\def\cC{{\cal C}}

\newtheorem{lem*}{Lemma}
\def\pn{\medskip\par\noindent}
\def\eps{\varepsilon}

\def\Lim{\mathop{\hbox{lim}}\limits}
\def\abs#1{\left \vert #1 \right \vert}
\def\Frac#1#2{{\displaystyle{\frac{#1}{#2}}}}
\def\[#1\]{\begin{equation}#1\end{equation}}
\def\$#1\${\begin{eqnarray}#1\end{eqnarray}}
\def\[#1\]{\begin{eqnarray} #1 \end{eqnarray}}
\def\phi{\varphi}

\def\RR{{\bf R}}
\def\bn{\begin{enumerate}}
\def\en{\end{enumerate}}
\def\ent#1#2{{\textstyle{{#1} \overwithdelims[] {#2}}}}

\def\pent#1#2{\pe{\frac{#1}{#2}}}

\def\pe#1{{\left \lbrack #1 \right \rbrack}} 

\def\equi{\mathop{\displaystyle{\simeq}}\limits}

\def\tiret{--- }
\def\pn{\medskip\par\noindent}
\def\vect{\hbox{\rm vect }}
\def\eps{\varepsilon}
\def\NN{{\rm I\!N}} 
\def\RR{{\rm I\!R}} 
\def\RR{{\bf R}}
\begin{document}
\pagestyle{myheadings}
\markboth{P. -V. Koseleff, D. Pecker}{{\em A polynomial parametrization of torus knots}Submitted}
\title{A polynomial parametrization of torus knots}
\author{P. -V. Koseleff \and D. Pecker}
\institute{UPMC Paris 6, 4, place Jussieu, F-75252 Paris Cedex 05, \\
\email{\tt\{koseleff,pecker\}@math.jussieu.fr}} 
\maketitle
\begin{abstract}
For every odd integer $N$ we give an explicit construction of a polynomial
curve $\cC(t) = (x(t), y (t))$, where $\deg x = 3$, 
$\deg y = N + 1 + 2\pent N4$ that has exactly $N$
crossing points 
$\cC(t_i)= \cC(s_i)$ whose parameters satisfy $s_1 < \cdots < s_{N} < t_1
< \cdots < t_{N}$. Our proof makes use of the theory of Stieltjes series and Pad{\'e} approximants.
This allows us an explicit polynomial parametrization of the torus knot $K_{2,N}$.
\end{abstract}
{\bf keywords:} {Polynomial curves, Stieltjes series, Pad{\'e} approximant, torus
  knots}    
\section{Introduction}
Let $N$ be an odd integer. 
We look for a parametrized curve $\cC(t)= (x(t),y(t))$ of minimal
lexicographic degree such that $\cC$ has exactly $N$ crossing points,
corresponding to parameters $(s_i, t_i)$ such that 
\begin{equation}\label{st}
\cC(s_i) = \cC(t_i), \ 
s_1< \cdots < s_N
< t_1 <\cdots <t_N.
\end{equation}
Here we look for curves with $\deg x=3$. 
As a consequence of B{\'e}zout
theorem, we have $\deg y \geq N+1$.
\begin{figure}[th]
\begin{center}
\psfrag{a1}{\hspace{-1.0cm}{\small $(s_1,t_1)$}}
\psfrag{a2}{\hspace{-.4cm}{\small$(s_2,t_2)$}}
\psfrag{a6}{{}}
\psfrag{a7}{\hspace{.4cm}{\small$(s_N,t_N)$}}
\centerline{\scalebox{.8}{\includegraphics{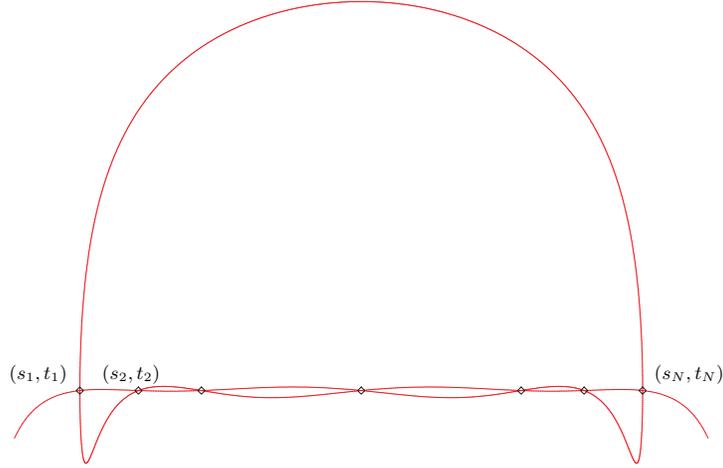}}}
\vspace*{8pt}
\caption{Curve of degree $(3, 19)$ in logarithmic scale in $y$}
\end{center}
\end{figure}
We have translated this problem into a problem on real roots of
certain real polynomials in one variable. 
In \cite{KP} we proved that if $N>3$, there is no solution with 
$\deg y = N+1$. We have computed the first examples and we have shown that
the minimal degrees are $\deg y = N + 1 + 2 \pent N4$ for $N=3, 5, 7$.  
\pn
The purpose of this paper is to give an explicit construction at any
order of such curves with $\deg y = N + 1 + 2 \pent N4$.
\pn
In section {\bf \ref{cheby}.}, we first recall some properties of the Chebyshev
polynomials. Our construction is based on certain relations in the space
spanned by some of these polynomials. 
\pn
The explicit construction is given in section {\bf \ref{pc}}. It involves
some particular polynomial basis whose existence is proved in section
{\bf \ref{pa}.}, using Stieltjes series theory 
and Pad{\'e} approximation theory (see \cite{BG}).
\pn
In section {\bf \ref{t2t6}.}, we will show that the algebraic relation
between $\cos 2\theta$ and $\cos 6\theta$ may be
seen as a Stieltjes series, namely some algebraic hypergeometric
function. We will recall some properties of these 
functions and their approximations by rational functions in section
{\bf \ref{pa}.}, the so-called Pad{\'e} approximants.  
\pn
The polynomial curves whose existence are proved are of interest for an 
explicit polynomial parametrizations of the $(2,N)$-type torus knot $K_{2,N}$ (see
\cite{Ad,KP,Mu,RS}). In section {\bf \ref{tk}.}, we give an explicit
parametrizations for the knots $K_{2,N}$.
They are symmetric with respect to the $y$-axis and of smaller degrees
than those already known.    
\section{Some properties of the Chebyshev polynomials}\label{cheby}
\begin{definition}[Monic Chebyshev polynomials]\label{tn}\\
If $t= 2 \cos \theta$, let $T_n(t)= 2 \cos ( n \theta )$ and
$V_n (t) = \Frac{\sin ((n+1) \theta)}{\sin \theta }$.
\end{definition}
$T_n$ and $V_n$ are both monic and have degree
$n$. It is convenient for our problem to consider them as
basis of $\RR[t]$. 
\pn
Looking for a polynomial curve $\cC(t)= (x(t),y(t))$ where $\deg x= 3$,
one can suppose that 
$$
x(t) = T_3(t), \ 
y(t) = T_m(t) + a_{m-1}T_{m-1}(t) + \cdots + a_1 T_1(t).
$$
In \cite{KP} (Lemma {\bf A}) we have shown that 
if $s \not = t$ are real numbers such that  $T_3(s)= T_3(t),$ then 
for any integer $k$ we have
\begin{equation}\label{tv}
\Frac{T_k (t) - T_k(s)}{t - s} =\Frac{2}{\sqrt 3} \sin \Frac{k
\pi}3 \, V_{k-1}(s+t) = \eps_k V_{k-1}(s+t).
\end{equation}
We proved the following:
\begin{proposition}\label{curve1}
Let $\eps_k = \Frac 2{\sqrt 3} \sin \Frac{k\pi}3 = V_{k-1}(1)$ and   
\begin{equation}\label{R}
R_m = \eps_m V_{m-1} + \eps_{m-1} a_{m-1} V_{m-2} + \cdots + \eps_1
a_1 V_0. 
\end{equation}
\tiret
If $R_m$ has exactly $N$ distinct roots $-1<u_1<\cdots<u_N<1$
and no other in $[-2,2]$, then 
$$\cC(t) = (T_3(t), \ T_m(t) + a_{m-1}T_{m-1}(t) + \cdots + a_1 T_1(t))$$
has exactly $N$ crossing points. \\
\tiret
Let $u_i = 2 \cos \alpha_i$, then 
\begin{equation}\label{sta}
s_i = 2 \cos(\alpha_i + \pi/3), \, t_i = 2 \cos(\alpha_i - \pi/3)
\end{equation}
are the parameters of the crossing points and satisfy
$$s_1 < \cdots < s_N<t_1<\cdots<t_N.$$ 
\end{proposition}
We look for polynomials $R_m$ in $\RR[t]$ having $N$ roots that are linear
combinations of the $V_k$, where $k$ is not equal to $2 \mod 3$. 
We will consider separately $E \subset \RR[t]$ spanned by  $V_{6k+1}$ and
$V_{6k+3}$ and $\tilde E$ spanned by the $V_{6k}$ and $V_{6k+4}$. We
first describe these vectorial spaces as direct sums:
\begin{lemma}\label{EE}
$
E = T_1 \cdot \RR[T_6] \oplus T_1\cdot T_2 \cdot \RR[T_6], \quad 
\tilde E = 1 \oplus T_3 \cdot E.
$
\end{lemma}
\begin{proof}
\tiret
From $\sin(x+y)-\sin(x-y) = 2 \cos(x) \sin(y)$
we deduce that for every integers $n$ and $p$, we have 
$$
V_{n+p} - V_{n-p} = V_{p-1} T_{n+1} .
$$
We thus deduce that $V_1 =T_1=t$, $V_3 = T_1 \cdot T_2$ and 
$$
V_{6k+1} - V_{6k-3} = T_1 \cdot T_{6k} , \
V_{6k+3}-V_{6k-5} = V_3 \cdot T_{6k} = T_1 \cdot T_2 \cdot T_{6k}.
$$
From $T_{6k} = T_{k}(T_6)$, we
deduce by induction that 
$$E = T_1 \cdot \RR[T_6] \oplus T_1 \cdot T_2\cdot \RR[T_6].$$ 
\tiret From $\sin(x+y)+\sin(x-y)=
2\cos(y)\sin(x)$, we get 
$$
V_{n+3}+V_{n-3} = T_3 V_{n}
$$
so
$$
V_{6k+6}+V_{6k} = T_3 V_{6k+3}, \, 
V_{6k+4}+V_{6k-2} = T_3 V_{6k+1}.
$$
As $V_0 = 1$ and $V_{-2}=-1$, we thus deduce by induction that $\tilde E = 1 \oplus T_3 \cdot E$.
\qed%
\end{proof}
\begin{definition}
Let us define for $k \geq 0$,  
\begin{eqnarray*}
\tilde W_{2k} = V_{6k},\, 
W_{2k} &=& V_{6k+1},\, W_{2k+1} = V_{6k+3}, \,
\tilde W_{2k+1} = V_{6k+4}.
\end{eqnarray*}
\end{definition}
We have $\deg W_n =2n+2\pent n2+1$ and 
$E = \vect(W_k, \, k\geq 0)$.\\
We have $\deg \tilde W_n =2n+2\pent{n+1}2$ and 
$\tilde E = \vect(\tilde W_k, \ k\geq 0)$.
\pn
Using the Pad{\'e} approximation theory, we will prove in section {\bf
  \ref{pa}} (p. \pageref{cn}) 
\begin{theorem}\label{2n}
There exists a sequence of odd polynomials $C_n$ in $E$ such that
\begin{eqnarray*}
\vect (W_0, \ldots, W_n) =\vect (C_0, \ldots, C_n),\, 
C_n = t^{2n+1}F_n, \, F_n(0)=1.  
\end{eqnarray*}
Furthermore $F_n(t)>0$ when $t \in [-2,2]$. 
\end{theorem}
We find, up to some multiplicative constant, 
\begin{eqnarray*}
C_{{0}}&=&t = W_0,\\
C_{{1}}&=&{t}^{3}=W_{{1}}+2\,W_{{0}},\\
C_{{2}}&=&{t}^{5} \left( {t}^{2}-6 \right)=W_{{2}}-10\,W_{{1}}-16\,W_{{0}},\\
C_{{3}}&=&{t}^{7} \left( {t}^{2}-9/2 \right)=W_{{3}}+7/2\,W_{{2}}-15\,W_{{1}}-21\,W_{{0}},\\
C_{{4}}&=&{t}^{9} \left( {t}^
{4}-12\,{t}^{2}+33 \right)=W_{{4}}-22\,W_{{3}}-56\,W_{{2}}+176\,W_{{1}}+231\,W_{{0}},\\
C_{{5}}&=&{t}^{11} \left( {t}^{4}-{\frac {
102}{11}}\,{t}^{2}+{\frac {234}{11}} \right)=
W_{{5}}+{\frac {52}{11}}\,W_{{4}}-40\,W_{{3}}-
{\frac {910}{11}}\,W_{{2}}+208\,W_{{1}}+260\,W_{{0}}.
\end{eqnarray*}
We deduce from theorem {\bf \ref{2n}} and lemma {\bf \ref{EE}} the
following useful result for the construction of the height function
$z(t)$ of our knots (see section {\bf \ref{tk}.}).  
\begin{corollary}\label{p2n}
The sequence 
$\tilde C_0=1, \, \tilde C_n = -\Frac 13 T_3 C_{n-1}$
of even polynomials in $\tilde E$ satisfies: 
$$
\vect(\tilde W_0, \ldots, \tilde W_n) =\vect(\tilde C_0,
\ldots, \tilde C_n), \, 
\tilde C_n = t^{2n} \tilde F_n ,\, \tilde F_n (0) = 1. 
$$
Furthermore $\tilde F_n(t)>0$ when $t\in [-2,2]$.
\end{corollary}
\section{Construction of the prescribed curves}\label{pc}
We will construct polynomials $R_m$ in $E$ with $N=2n+1$ real roots
in $[-1,1]$ and no other roots in $[-2,2]$. They will be chosen as a
slight deformation of $C_n$. 
Let us first show properties of the polynomials $C_n$. 
\begin{lemma}\label{ap}
Let $0 < u_1 < \cdots < u_n <1$ be real numbers.
For $\eps$ being small enough, 
\bn
\item there exists a unique $(a_0, \ldots, a_{n-1})$ such that
  $\{0,\pm \eps u_1, \ldots, \pm \eps u_n\}$ are roots in $[-1,1]$ of  
$$
A_n(\eps) = C_n  + a_{n-1} C_{n-1} + \cdots + a_0 C_0.
$$
\item $\{0,\pm \eps u_1, \ldots, \pm \eps u_n\}$ are the only real
  roots in $[-2,2]$ of $A_n(\eps)$.
\en
\end{lemma}
\begin{proof}
Looking for $A_n(\eps)$ with roots $0$ and $\pm \eps u_i$ is equivalent to
the linear system
$$
\left ( 
\begin{array}{cccc}
C_{0}(\eps u_1) & C_{1}(\eps u_1) & \cdots & C_{n-1}(\eps u_1) \\
C_{0}(\eps u_2) & & \cdots & C_{n-1}(\eps u_2)\\
\vdots & & & \vdots\\
C_{0}(\eps u_n) & C_{1}(\eps u_n) & \cdots &C_{n-1}(\eps u_n)
\end{array}
\right ) 
\left ( 
\begin{array}{c}
a_0 \\ a_1 \\  \vdots \\ a_{n-1}
\end{array}
\right )
=
-
\left ( 
\begin{array}{c}
C_n(\eps u_1) \\ C_n(\eps u_2) \\  \vdots \\ C_n(\eps u_n)
\end{array}
\right )
$$
whose determinant is 
$
\eps^{(1+3+\cdots+2n-1)} 
\left |
\begin{array}{cccc}
u_1 F_{0}(\eps u_1) & u_1^{3} F_{1}(\eps u_1) & 
\cdots & u_1^{2n-1}F_{n-1}(\eps u_1) \\
u_2 F_{0}(\eps u_2) & & \cdots & u_2^{2n-1}F_{n-1}(\eps u_2)\\
\vdots & & & \vdots\\
u_n F_{0}(\eps u_n) & u_n^{3} F_{1}(\eps u_n) & \cdots &
u_n^{2n-1} F_{n-1}(\eps u_n)
\end{array}
\right |. 
$\\
It is equivalent to the classical 
Vandermonde-type determinant when $\eps \rightarrow 0$:
$$ 
\eps^{n^2} 
\left |
\begin{array}{cccc}
u_1 & u_1^{3} & \cdots & u_1^{2n-1} \\
u_2 & & \cdots & u_2^{2n-1}\\
\vdots & & & \vdots\\
u_n & u_n^{3} & \cdots & u_n^{2n-1}
\end{array}
\right |
= \eps^{n^2} u_1 \cdots u_n \prod_{1 \leq i < j \leq n} (u_j^2 -
u_i^2) \not = 0.
$$
Therefore, this system has a unique
solution. 
\pn 
\tiret Using Cramer formulas, we get  
$
a_k(\eps) = \Frac{G_k(\eps)}{G_{n}(\eps)}
$
where 
\begin{eqnarray*}
G_k (\eps) = \det (C_{i_j} (\eps u_i))_{1 \leq i,j \leq n}
\end{eqnarray*}
and $i_1 < i_2 < \cdots < i_n \in \{0,\ldots,n\} - \{k\}$.
We get  
\begin{eqnarray*}
G_k(\eps) &=& 
\eps^{2(i_1 + \cdots + i_n)+n}
\left | 
\begin{array}{cccc}
u_1^{2i_1+1} F_{i_1}(\eps u_1) & u_1^{2i_2+1} F_{i_2}(\eps u_1) & 
\cdots & u_1^{2i_n+1}F_{i_n}(\eps u_1) \\
u_2^{2i_1+1} F_{i_1}(\eps u_2) & & \cdots & u_2^{2i_n+1}F_{i_n}(\eps u_2)\\
\vdots & & & \vdots\\
u_n^{2i_1+1} F_{i_1}(\eps u_n) & u_n^{2i_2+1} F_{i_2}(\eps u_n) & \cdots &
u_n^{2i_n+1} F_{i_n}(\eps u_n)
\end{array}
\right |\\
&\equi_{\eps \to 0}&
\eps^{(n+1)^2-(2k+1)} 
\left | 
\begin{array}{cccc}
u_1^{2i_1+1} & u_1^{2i_2+1} & \cdots & u_1^{2i_n+1} \\
u_2^{2i_1+1} & & \cdots & u_2^{2i_n+1}\\
\vdots & & & \vdots\\
u_n^{2i_1+1} & u_n^{2i_2+1} & \cdots & u_n^{2i_n+1}
\end{array}
\right |. 
\end{eqnarray*}
\tiret
We thus deduce that $a_k(\eps) = {\cal O}(\eps^{2(n-k)})$ and therefore
$$\lim_{\eps \to 0} A_n(\eps) = C_n = t^{2n+1} F_n.$$
Let $A_n (\eps)= t\prod_{i=1}^n (t^2-\eps^2 u_i^2) D_n(\eps)$.
We deduce that $\Lim_{\eps \to 0} D_n = F_n$. 
Let $\eps$ be small enough, we get $D_n(t)>0$ for $t \in[-2,2]$ 
because of the compactness of $[-2,2]$.
\qed%
\end{proof}
\begin{proposition}\label{curve2}
Let $N=2n+1$ be an odd integer. There exists a curve $\cC(t)=(x(t),y(t))$,
where $\deg x = 3$ and $\deg y = N+2 \pent N4 +1$, such that $\cC$ has
exactly $N$ crossing points  
corresponding to parameters $(s_i, t_i)$ such that 
\begin{equation}
\cC(s_i) = \cC(t_i), \ 
s_1< \cdots < s_N
< t_1 <\cdots <t_N. \nonumber
\end{equation}
\end{proposition}
\begin{proof}
Let $N=2n+1$. Let us choose $\eps$ and $0<u_1<\cdots<u_n$, such that  
there exists a polynomial $A_n(\eps) \in \vect(W_0, \ldots, W_n)$ 
having exactly $N$ distinct roots $\{0, \pm \eps u_1, \ldots, \pm \eps
u_n\}$ in $[-2,2].$ It has degree $\deg W_n = 2n + 2 \ent n2 +1 =
N+2\pent N4  = m-1$.  We have
$$
 A_n = W_n + \cdots + a_0  = V_{m-1} + \eps_{m-1} a_{m-1} V_{m-2}
+\cdots +  a_2 V_1.
$$
Using proposition {\bf \ref{curve1}}, the curve 
$$
x(t) = T_3(t) , \ y(t) = \eps_m T_m(t) + a_{m-1}T_{m-1}(t) + \cdots + a_2 T_2(t)
$$
has the required properties. 
\qed%
\end{proof}
\begin{example}[$N=9$]
We chose 
$$y(t) = 
-{\Frac {27}{10}}\,T_{{14}} 
+10\,T_{{12}}
-23\,T_{{10}}
+42\,T_{{8}}
-64\,T_{{6}}
+85\,T_{{4}}
-100\,T_{{2}}
+112.$$
The roots of $Q$ are $u_0 = 0, \pm u_1 =  .355, \pm u_2 =
.584, \pm u_3 = .785, \pm u_4= 1.073$.
We obtain a polynomial parametrization of degree $(3,14)$. 
Note that we choose $u_4>1$ for a nicer picture (see
fig. \ref{fig9}). Note that the parameters of the crossing points
satisfy $s_1 < \cdots < s_9 < t_1 < \cdots <t_9$.  
\begin{figure}[th]
\begin{center}
\begin{tabular}{ccc}
{\scalebox{.4}{\includegraphics{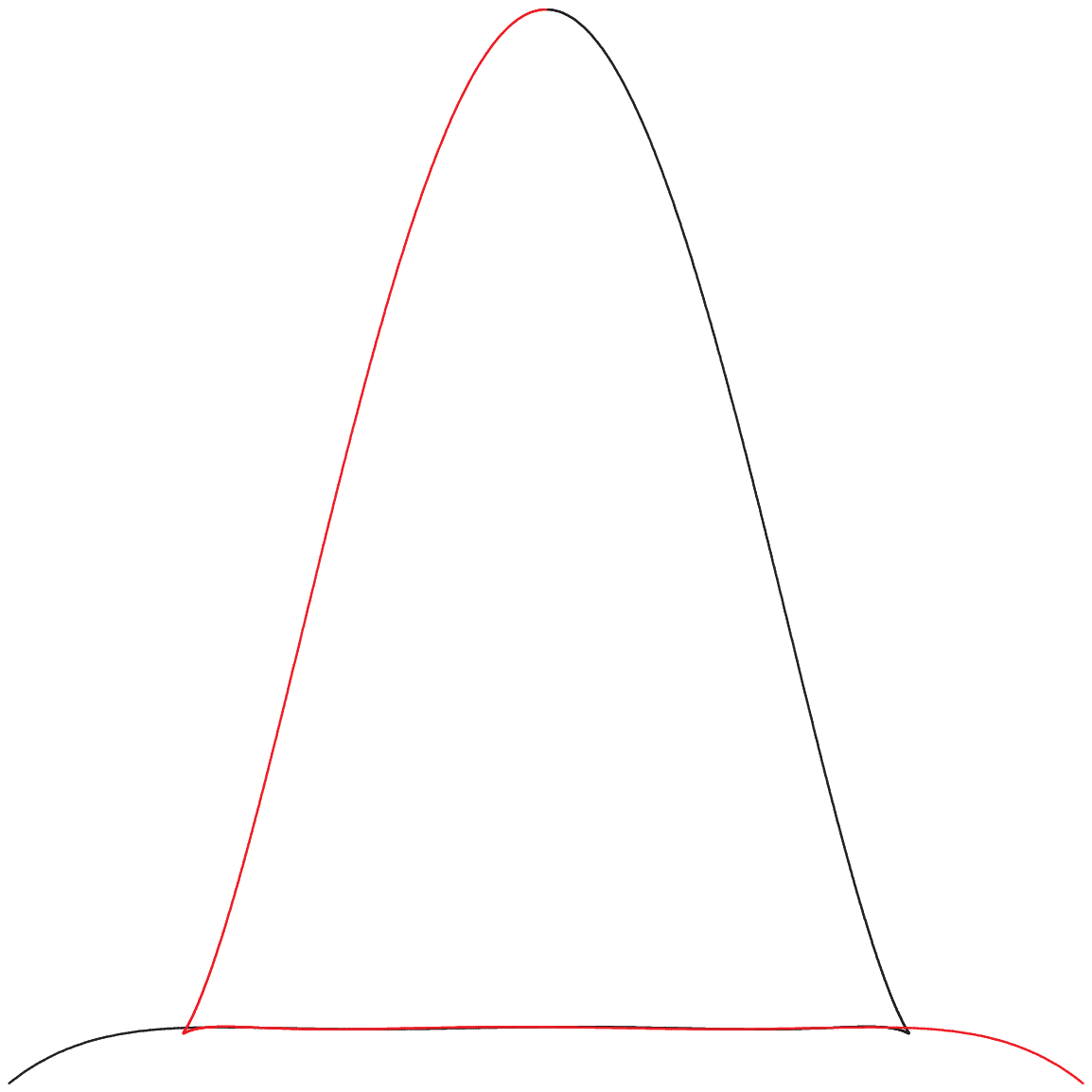}}}
&\quad&
{\scalebox{.4}{\includegraphics{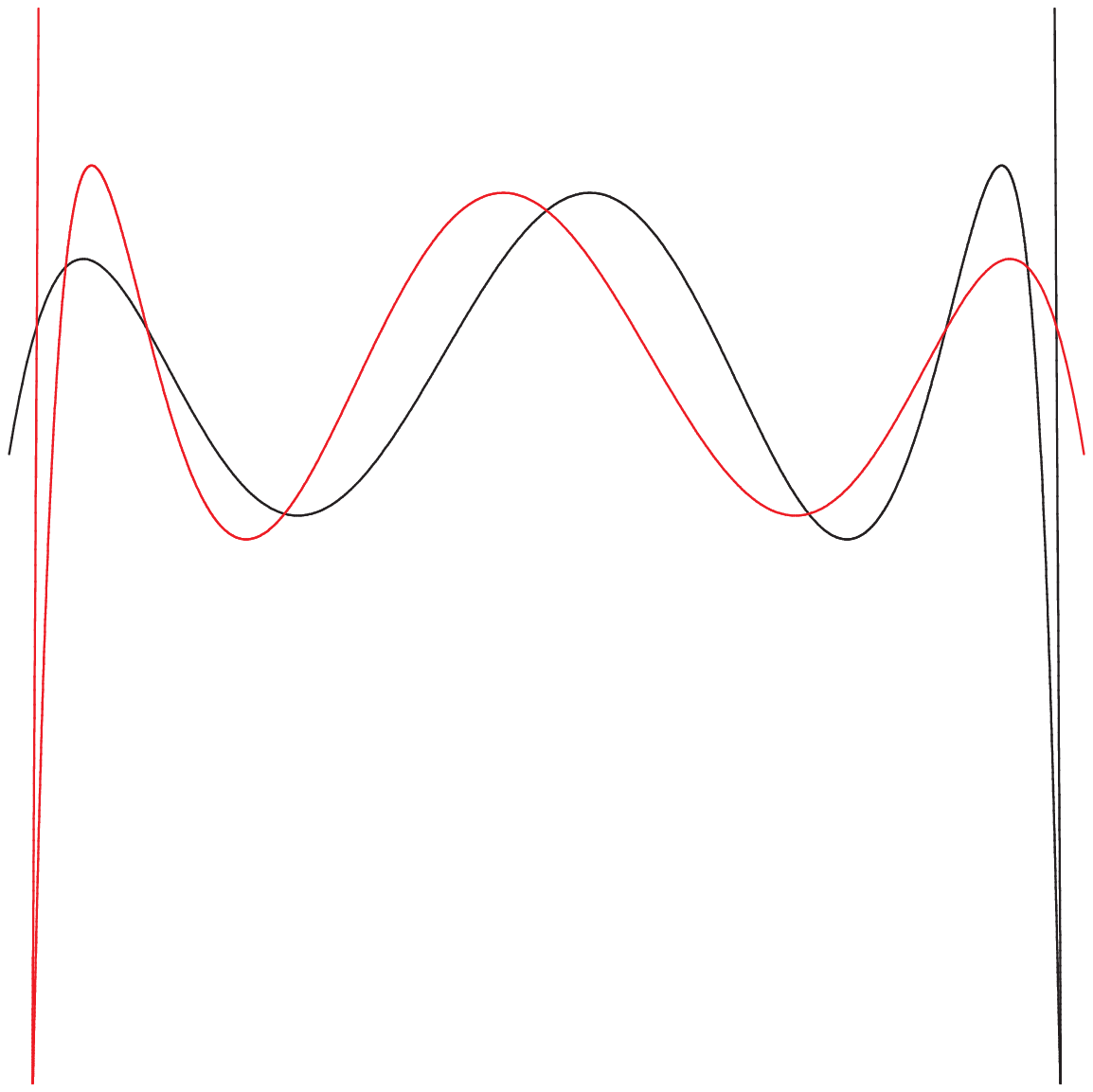}}}\\
Bottom view &  & Zoom on the bottom view\\
{\scalebox{.4}{\includegraphics{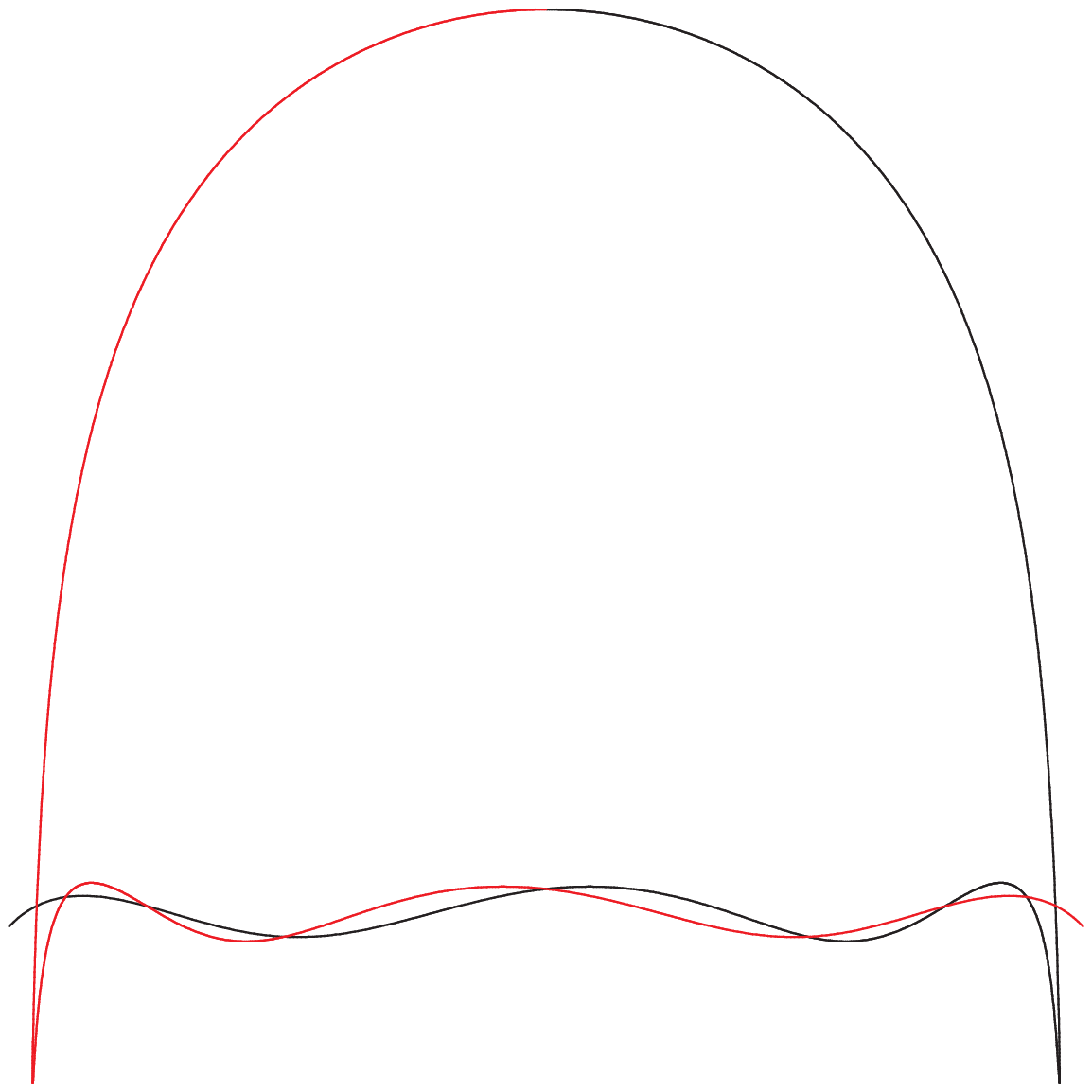}}}
&\quad&
{\scalebox{.4}{\includegraphics{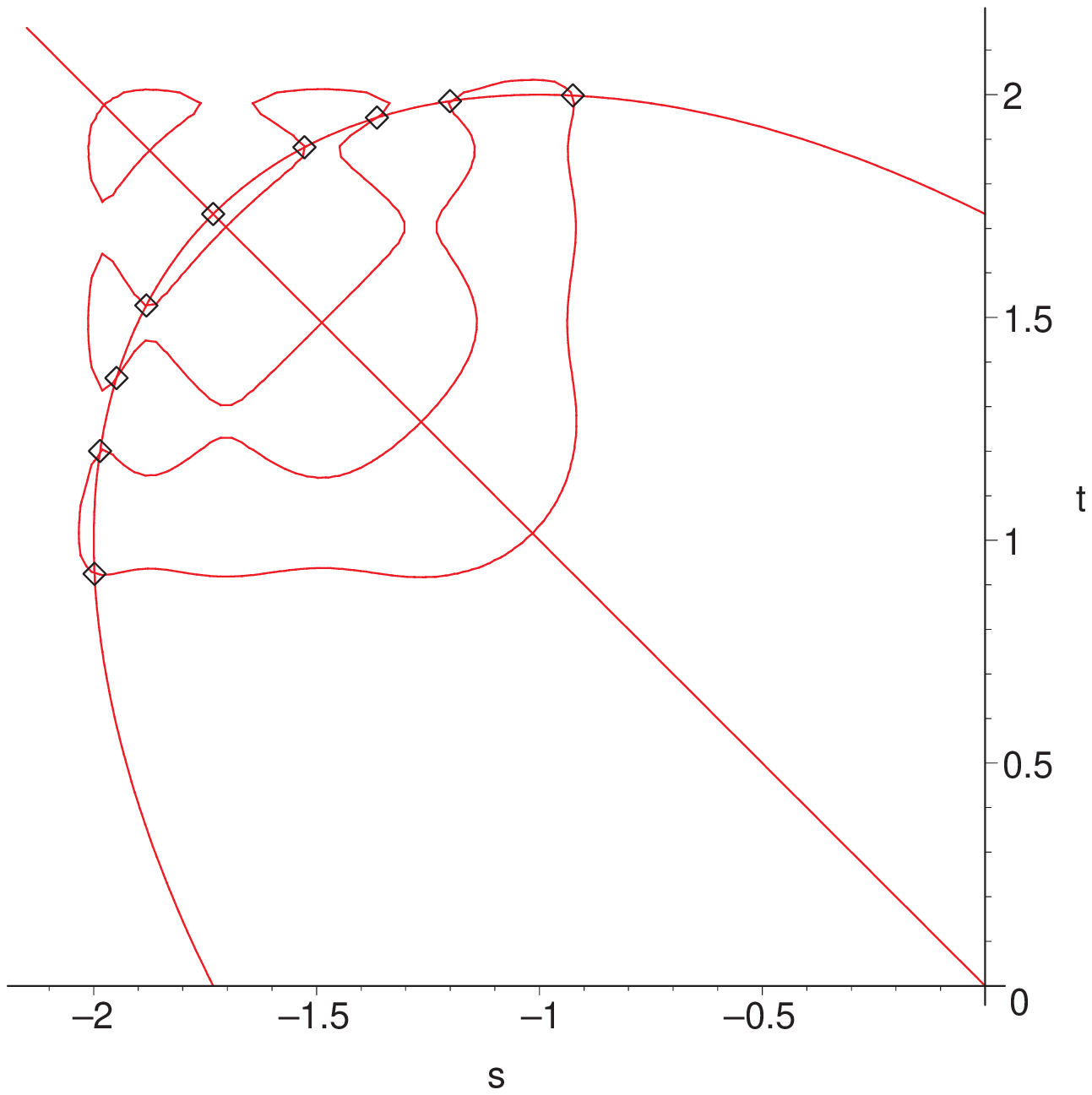}}}\\
logarithmic scale &  & $\Frac{x(s)-x(t)}{s-t}=0, \
\Frac{y(s)-y(t)}{s-t}=0, s<t$
\end{tabular}
\end{center}
\vspace{8pt}
\caption{$N=9$. Curve of degree $(3,14)$\label{fig9}}
\end{figure}
\end{example}
\section{Construction of the torus knots}\label{tk}
If $N=2n+1$ is odd, the torus knot $K_{2,N}$ of type $(2,N)$ is the
boundary of a  Moebius band  twisted $N$ times (see \cite{Ad,Mu} and
fig. (\ref{fig1})).
\begin{figure}[th]
\centerline{
{\scalebox{.8}{\includegraphics{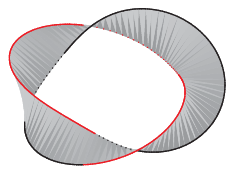}}} \quad
{\scalebox{.8}{\includegraphics{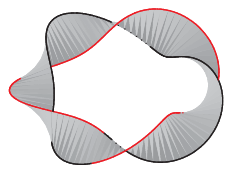}}} \quad
{\scalebox{.8}{\includegraphics{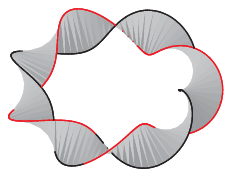}}}}
\vspace{10pt}
\centerline{
{\scalebox{.8}{\includegraphics{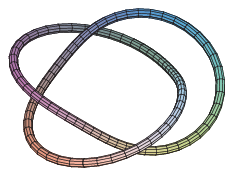}}} \quad
{\scalebox{.8}{\includegraphics{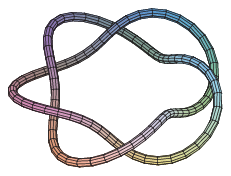}}} \quad
{\scalebox{.8}{\includegraphics{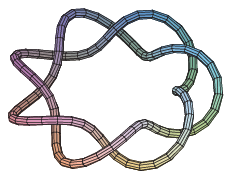}}}}
\vspace*{8pt}
\caption{$K_{2,N}$, $N=3,5,7$.\label{fig1}}
\end{figure}
The purpose of this section is to give an explicit construction of a
polynomial curve $\cC(t) = (x(t), y(t), z(t))$ that
is equivalent (in the one point compactification $ {\bf S}^3$ of the
space $ {\RR}^3$) to the torus knot $K_{2,N}$.  

Vassiliev (see \cite{Va}) proved that any non-compact knot type can be
obtained from a polynomial embedding
$ \  t \mapsto ( f(t), g(t), h(t) ),  t \in \RR$, 
using the Weierstrass approximation theorem.

Shastri \cite{Sh} gave a detailed proof of this theorem,
and a simple polynomial parametrizations of the trefoil and of the
figure eight knot.
A. Ranjan and R. Shukla \cite{RS} have found small degree
parametrizations for $K_{2,N}, N$ odd. They proved that these knots
can be attained from polynomials of degrees  $ (3, 2N-2, 2N-1).$

In \cite{KP}, we proved that it is not  possible to attain the torus
knot $K_{2,N}$ with polynomial of degrees $(3,N+1,m)$ when $N>3$. We
gave explicit parametrization of degrees $(3,N+ 2\pent
N4+1,N+2\pent{N+1}4)$ for $N=3,5,7,9$ and showed that they were of
minimal lexicographic degree for $N\leq 7$. 
\pn
A sufficient condition is to construct a parametrized curve 
$\cC(t)=(x(t),y(t),z(t))$
such that $(x(t),y(t))$ has exactly $N=2n+1$ crossing points corresponding to
parameters   
$$s_1 < \cdots < s_N < t_1 < \cdots < t_N$$ 
and such that
$$x(t_i)=x(s_i), \ y(t_i)=y(s_i), 
\, (-1)^i(z(t_i)-z(s_i))>0, \ i = 1, \ldots , N.$$
We look first for minimal degree in $x$. 
$x(t)$ must be nonmonotonic and therefore has degree at least 2. In case when 
$\deg x=2$, we would have constant $t_i + s_i$ and not the condition
(\ref{st}).  We will give a construction for $\deg x=3$. 
\begin{proposition}\label{kn}
For any odd integer $N=2n+1$, there exists 
a curve $\cC (t) = (x(t),y(t),z(t))$ of degree 
$(3,N+ 2\pent N4+1,N+2\pent{N+1}4)$ 
such that the curve $(x(t),y(t))$ has exactly $N$ crossing points 
$$x(t_i)=x(s_i), \ y(t_i)=y(s_i), \, 
s_1< \cdots < s_N < t_1 <\cdots <t_N
$$ 
and 
$$(-1)^i (z(t_i)-z(s_i)) >0 , \ i = 1, \ldots , N.$$
\end{proposition}
\begin{proof}
\tiret
Following the construction of section {\bf \ref{pc}.},
we first choose $\eps$ to be small enough and 
$0 < c_1 < \cdots < c_n < \eps < 1/2$, such that 
$$
\left \vert
\begin{array}{cccc}
C_{0}(c_1) & C_{1}(c_1) & \cdots & C_{n-1}(c_1) \\
C_{0}(c_2) & & \cdots & C_{n-1}(c_2)\\
\vdots & & & \vdots\\
C_{0}(c_n) & C_{1}( c_n) & \cdots &C_{n-1}(c_n)
\end{array}
\right \vert \not = 0.
$$
Consider 
\begin{eqnarray*}
u_{n+1} = 0, \, u_i = 2\cos \alpha_i  = - c_{n+1-i},\,
u_{n+1+i} = 2\cos \alpha_{n+1+i} = c_i, 
\, i = 1,\ldots, n.
\end{eqnarray*}
We thus have $-1< u_1 <\ldots < u_N <1$. 
Let 
$$s_i = 2 \cos(\alpha_i + \pi/3), \, t_i = 2 \cos(\alpha_i - \pi/3),
\, i =1, \ldots N.$$
\tiret
Using the proposition (\ref{curve2}), there is a polynomial 
$$
y(t) = T_m(t) + a_{m-1}T_{m-1}(t) + \cdots + a_1 T_1(t), 
$$
of degree $m = N + 2 \pent N4+1$ such that 
$$x(t_i)=x(s_i), \ y(t_i)=y(s_i), \, i =1, \ldots N.$$ 
\tiret As for lemma {\bf \ref{ap}}, there exists a unique 
$(b_0, \ldots, b_{n})$ 
such that 
$$
B_n = b_n \tilde C_n + b_{n-1} \tilde C_{n-1} + \cdots + b_{0} \tilde C_0
$$
satisfies $B_n(u_i) = (-1)^i, \, i=1,\ldots N$. Namely, 
$(b_0, \ldots, b_{n})$ is the solution of the system 
$$
 b_n \tilde C_n (u_i) + b_{n-1} \tilde C_{n-1}(u_i) + \cdots + b_{0}
 \tilde C_0(u_i) = (-1)^i, \, i=1,\ldots, N. 
$$
Because $u_i = -u_{N+1-i}$ and $\tilde C_k$ are even polynomials, the
system is equivalent to 
$$
\left ( 
\begin{array}{cccc}
\tilde C_{0}(u_{n+1}) & \tilde C_{1}(u_{n+1}) & \cdots & \tilde C_{n}(u_{n+1}) \\
\tilde C_{0}(u_{n+2}) & & \cdots & \tilde C_{n}(u_{n+2})\\
\vdots & & & \vdots\\
\tilde C_{0}(u_N) & \tilde C_{1}(u_N) & \cdots &\tilde C_{n}(u_N)
\end{array}
\right ) 
\left ( 
\begin{array}{c}
b_0 \\ b_1 \\  \vdots \\ b_{n}
\end{array}
\right )
=
\left ( 
\begin{array}{c}
(-1)^{n+1} \\ (-1)^{n+2}\\  \vdots \\ (-1)^{N}
\end{array}
\right ).
$$
From $\tilde C_0 = 1$, $\tilde C_n = -\Frac 13 T_3 C_{n-1}$ and
$C_k(u_{n+1})=0$, we deduce that the determinant of the previous system is
$$
\pm\Frac{1}{3^n} T_3(c_1) \cdots T_3(c_n)
\left \vert
\begin{array}{cccc}
C_{0}( c_1) & C_{1}(c_1) & \cdots & C_{n-1}( c_1) \\
C_{0}( c_2) & & \cdots & C_{n-1}(c_2)\\
\vdots & & & \vdots\\
C_{0}(c_n) & C_{1}( c_n) & \cdots &C_{n-1}(c_n)
\end{array}
\right \vert \not = 0.
$$
$B_n$ is a linear combination of $(\tilde W_n,\ldots,
\tilde W_0)$ and it has degree $m'=N+2\pent{N+1}4$: 
$$
B_n = b'_{m'} V_{m'} + \cdots + b'_0 V_0. 
$$
Consider now $z(t) = \eps_{m'+1} b'_{m'} T_{m'+1} + \cdots + \eps_1 b'_0
T_1$, we have using eq. (\ref{tv}):
$$
\Frac{z(t_i) -z(s_i)}{t_i -s_i}=B_n(u_i) = (-1)^i. 
$$ 
Because $t_i > s_i$ we deduce that $z(t_i)-z(s_i)$ has alternate signs.
\qed%
\end{proof}
\section{$T_2$ as a power series of $T_6+2$}\label{t2t6}
Looking for identities in the vectorial space $\RR[T_6] + T_2 \cdot
\RR[T_6]$, we show first some relation between $T_2$ and $T_6$.
\begin{lemma}\label{t1t3}
For $t \in [-1,1]$, we have 
$$
T_2+2 =4\sin^2 \left (\Frac 13   \arcsin \sqrt{\Frac{T_6+2}4} \right )
. 
$$
\end{lemma}
\begin{proof}
Let $t \in [-1,1]$ and $x \in [\pi/3,2\pi/3]$ such that $t=2\cos x$. \\
We get $3x-\pi \in [0,\pi]$ and $\cos (3x -\pi)=-\Frac 12 T_3$ so
$$x=
\Frac{\pi}3+\Frac 13 \arccos \left (-\Frac{T_3}2\right ) 
= \Frac{\pi}3 + \Frac 13 \left (\Frac{\pi}2 + \arcsin \Frac{T_3}2 \right) 
= \Frac{\pi}2 + \Frac 13 \arcsin \Frac{T_3}2.$$ 
We thus have 
$$
T_1 = 2\cos \left (\Frac{\pi}2 +\Frac 13 \arcsin \Frac{T_3}2  \right
) = - 2\sin\left (\Frac 13   \arcsin \Frac{T_3}2 \right ). 
$$
We thus deduce the lemma from $T_2 = T_1^2 -2$ and $T_6 = T_3^2 -2$.
\qed%
\end{proof}
\begin{lemma}
Let $\phi(u)  = 4\sin^2 \left (\Frac 13   \arcsin \sqrt{u} \right )$. 
For $u \in [0,1]$, we have 
$$
\phi(u) = 
\sum_{n \geq 1} \phi_n u^n 
\hbox{ where }
\phi_1 = \Frac 49, \, 
\phi_{n+1} = \Frac 29 \Frac{(3n+1)(3n-1)}{(n+1)(2n+1)}\phi_n.
$$
\end{lemma}
\begin{proof}
We have 
$\phi(u)  
= 2 - 2 \cos \left ( \Frac 23 \arcsin \sqrt u \right )
.$
We deduce that 
\[\label{sysd}
\left \{
\begin{array}{rcl}
\phi(u) 
&=& - 2 A+2\\[5pt]
\Frac {d}{d{u}}\phi(u) 
&=& \Frac 23 B\\[5pt]
\Frac{d^{2}}{d{u}^{2}}\phi(u) &=&
- \Frac 29 \Frac 1{(u-u^2)}  A+
\Frac 13 \Frac{(2u-1)}{(u-u^2)} B
\end{array}
\right .
\]
where 
$A=
\cos \left (
  \Frac 23 \arcsin \sqrt u \right )$ and
$B=\Frac{\sin \left (\Frac 23 \arcsin \sqrt u \right )}{\sqrt{u-u^2}}$.
\pn
Eliminating $A$ and $B$ from system (\ref{sysd}), we find that
\[\label{ode}
-4+2\,\phi \left( u \right) + 9 \left (1-2\,u \right) {\frac {d}{du}}\phi
 \left( u \right) + 18\left( u-u^2 \right ) 
{\frac {d^{2}}{d{u}^{2}}}
\phi \left( u \right)  = 0. 
\]
$\phi$ has a power series expansion and we get from (\ref{ode})
$$
\phi_0 = 0 , \, \phi_1 = \Frac 49, \, 
\phi_{n+1} = \Frac 29 \Frac{(3n+1)(3n-1)}{(n+1)(2n+1)}\phi_n.
$$
\qed%
\end{proof}
\begin{remark}
There is no need to know explicitely $\phi$ with the lemma
{\bf \ref{t1t3}}. One can see from $4u=v(v-3)^2$ 
that $\phi$ is an algebraic
function. It is therefore the solution of a differential equation we
can find using Euclid algorithm. Recursion formula for the $\phi_n$
and the differential equation can be easely obtained using 
the {\sc Maple} package {\tt gfun} (see \cite{SZ}).
\end{remark}
\begin{definition}
Let $\Delta$ defined by $\Delta f_n = f_{n+1}-f_{n}$. We say that 
$f_n$ is totally monotone when 
for every integer $k$ and every $n \geq 1$, we have
$$(-1)^k \Delta^k f_n > 0.$$ 
\end{definition}
\begin{example}\label{ex1}
\tiret Let $f_n=\exp(-n)$. We get $(-1)^k \Delta^k f_n = f_n (1-1/e)^k$.
\pn
\tiret Let $f_n = \Frac 1n$. We get $(-1)^k \Delta^k f_n = f_n \Frac
1{{n+k \choose k}}.$
\pn They are both totally monotone.
\end{example}

\def\NN{{\rm I\!N}} 
\begin{proposition}
$\phi_n$ is totally monotone. 
\end{proposition}
\begin{proof} 
We will show that $$(-1)^k \Delta^k \phi_n = \phi_n \Frac{P_k(n)}{(n+1)
  \cdots (n+k) \cdot  (2n+1)\cdots (2n+2k-1)}>0.$$
\tiret
We get 
$$\Delta \phi_n = \phi_{n+1} - \phi_n = \phi_n \left (
\Frac{2
  (3n-1)(3n+1)}{9(n+1)(2n+1)} - 1\right ) =
- \phi_n \Frac{3 n+11/9}{(n+1)(2n+1)}. $$
Suppose now that
$(-1)^k \Delta^k\phi_n = \phi_n \Frac{P_k(n)}{(n+1)
  \cdots (n+k) \cdot  (2n+1)\cdots (2n+2k-1)}.$
We thus deduce
\begin{eqnarray*}
(-1)^{k+1} \Delta^{k+1}\phi_n &=& 
-\Delta \left [
\phi_n \Frac{P_k(n)}{(n+1)
  \cdots (n+k) \cdot  (2n+1)\cdots (2n+2k-1)}
\right ]\nonumber\\
&=& 
\phi_n \Frac{P_k(n)}{(n+1) \cdots (n+k) \cdot  (2n+1)\cdots (2n+2k-1)}
- \nonumber\\
&&\quad
\phi_{n+1} \Frac{P_k(n+1)}
{(n+2) \cdots (n+k+1) \cdot  (2n+3)\cdots (2n+2k+1)}\nonumber\\
&=&
\phi_n \Frac{(n+k+1)(2n+2k+1) P_k(n) - 2(n^2-1/9) P_{k}(n+1)}
{(n+1) \cdots (n+k+1) \cdot  (2n+1)\cdots (2n+2(k+1)-1)}.
\end{eqnarray*}
We thus obtain 
$$
(-1)^k \Delta^k \phi_n = \phi_n \Frac{P_k(n)}{(n+1)
  \cdots (n+k) \cdot  (2n+1)\cdots (2n+2k-1)},
$$
where $P_0  = 1$ and 
$$
P_{k+1} (n) = {(n+k+1)(2n+2k+1) P_k(n) - 2(n^2-1/9) P_{k}(n+1)}. 
$$
\tiret
We will show now by induction that $P_k = a_k X^k + \cdots + a_0$
where $a_k >0$.  Suppose it is true for a given $k$, we thus deduce that
\[
P_{k+1} &=& (X+k+1)(2X+2k+1) (a_k X^k + a_{k-1} X^{k-1} + \cdots ) -
\nonumber\\
&&\quad \quad 
2(X^2-1/9) (a_k X^k + (a_{k-1} + k a_k) X^{k-1} + \cdots ) \nonumber\\
&=& 2 a_k X^{k+2} + \left ((4k+3) a_k + 2 a_{k-1}) \right ) X^{k+1} + \cdots 
- \nonumber\\
&&\quad \quad 
\left [ 2 a_k X^{k+2} + (2 a_{k-1} + 2 k a_k) X^{k+1} + \cdots \right
] \nonumber\\ 
&=&  
(2k+3) a_k X^{k+1} + \cdots . 
\]
$P_k$ is a polynomial of degree $k$ whose leading coefficient is 
$1 \cdot 3 \cdots (2k+1)$.
\def\cP{{\cal P}}
\pn
\tiret
Let us prove now by induction the following
$$
(-1)^i P_k(-i) > 0, \, i = 0, \ldots, k.
$$
This is true for $k=0$.
\pn
Suppose now it is true for $P_k$. Intermediate values theorem says that
$P_k$ has exactly $k$ real roots in $]-k,0[$, so
$P_k (x) >0$ when $x\geq 0$ or when $x+k\leq 0$.
\pn
Let us compute
$$
P_{k+1}(0) = (k+1) (2k+1) P_{k}(0) + 2/9 P_k(1) > 0 
$$
For $i = 1, \ldots, k$~: 
$$
(-1)^i P_{k+1}(-i) = (k-i+1)(2(k-i)+1) (-1)^i P_k (i) + 2 (i^2 - 2/9)
(-1)^{i-1} P_k(-(i-1)) > 0. 
$$
For $i = - (k+1)$ we get
$$
(-1)^{k+1} P_{k+1}(-(k+1)) = 0 - 2 ((k+1)^2 - 2/9) (-1)^{k+1} P_k(-k) > 0
$$
We thus deduce that $(-1)^i P_{k+1}(-i) >0$ for $i = 0,
\ldots , k+1$. 
\pn
\tiret
We thus deduce that $P_k$ has exactly $k$ roots in $]-k,0[$ so
$P_k(n)$ is nonnegative for any integer $n$. 
\qed%
\end{proof}
\begin{definition}\label{stieltjes}
$f(z) = \sum_{n\geq 1} f_n z^n$ is a Stieltjes series
if for every $n \geq 1$ and $m \geq 0$, one has
$$
\left | 
\begin{array}{cccc}
f_n & f_{n+1} & \cdots & f_{n+m}\\
f_{n+1}& f_{n+2} & \cdots & f_{n+m+1}\\
\vdots & & & \vdots\\
f_{n+m} & f_{n+m+1} & \ldots & f_{n+2m}
\end{array}
\right |>0.
$$
\end{definition}
\begin{remark}
This last condition is related to the problem of Hamburger moments. It
is the Stieltjes condition. 
The totally monotonicity is related to the
Hausdorff condition (see \cite{Ha}).
\end{remark}
\def\d{\hbox{\rm d}\,}
The Hausdorff condition and the Stieltjes condition are equivalent 
if the series is not a rational function (see \cite{BG}, p. 194 and the proof of
Sch{\"o}nberg, \cite{Wa}, p. 267 or \cite{Sc}).
We thus deduce that
\begin{theorem}
$\phi(z)= \sum_{n\geq 1} \phi_n z^n$ is a Stieltjes series.
\end{theorem}
\begin{proof}
$\phi(u)$ is an algebraic function that satisfies 
$4u=\phi(\phi-3)^2$. Suppose that $\phi = p/q$ where $p(u)$ and $q(u)$
are relatively prime polynomials in $u$, then we would have 
$4uq^3-p^3+6p^2q-9pq^2 =0$ and $p$ would divide $u$ and $q$
would divide $1$. We would have $\phi(u) = \lambda u$ and it is not
the case. Thus $\phi$ is not a rational function and is therefore a Stieltjes function. 
\qed%
\end{proof}
\begin{remark}
In example (\ref{ex1}), the sequence $\exp(-n)$ is totally monotonic. But 
$\sum_n \exp(-n) z^n = \Frac 1{1-e\cdot z}$ is a rational function and the 
condition (\ref{stieltjes}) does not hold. 
\end{remark}
\begin{remark}
$\phi(u)= 2-2 F(1/3,-1/3,1/2;u)$ where $F(a,b,c;z)$ is the hypergeometric
function. It results from eq. (\ref{ode}) that is known as the
hypergeometric equation (\cite{BG,Wa})
$$
\left( u-u^2 \right ) {\frac {d^{2}}{d{u}^{2}}} f\left( u \right)  
+ \left (c -(1+a+b)\,u \right) {\frac {d}{du}}f \left( u \right)
- ab \,f \left( u \right) 
= 0. 
$$
for $\phi - 2 = -2f$, $a=-b=\Frac 13, c=\Frac 12$.
\end{remark}
\section{Pad{\'e} approximation}\label{pa}
Rational approximations of Stieltjes series have remarkable
properties. Let us remind the following construction of Pad{\'e} approximants:
\begin{theorem}[Pad{\'e} approximant]\label{pade-t}
Let $f(x) = \sum_{k \geq 1} f_k x^k$ be a Stieltjes series and consider two
integers $m\leq n$. 
There is a unique solution $(P_n, Q_m) \in \RR_n[x] \times
\RR_{m}[x]$, such that  
\begin{equation}
Q_m(0)=1,\, P_n - f Q_{m} = 0 \mod{x^{n+m+1}}.\label{pade}
\end{equation}
Furthermore we have $\deg P_n=n$ and $\deg Q_m=m$. 
\end{theorem}
\begin{proof}
Let us write 
$$
P_n = p_0 + \cdots + p_n x^n, \, Q_m = q_0 + q_1 x + \cdots + q_m x^n. 
$$
Eq. (\ref{pade}) gives 
\begin{eqnarray}
&&\left \{
\begin{array}{rcl}
p_0 &=& f_0, \\
p_1 &=& f_0 q_1+f_1 q_0\\
& \vdots & \\
p_n & = & f_{n-m} q_m +f_{n-m+1} q_{m-1} + \cdots + f_{n}q_0,
\end{array}\right . \label{nm}\\
&&
\left \{
\begin{array}{rcl}
0 &= & f_{n-m+1} q_{m} + f_{n-m+2} q_{m-1} + \cdots + f_{n+1} q_0  \\
0 &= & f_{n-m+2} q_{m} + f_{n-m+3} q_{m-1} + \cdots + f_{n+2} q_0  \\
&\vdots& \\
0 &=& f_{n} q_m + f_{n+1} q_{m-1} + \cdots + f_{m+n} q_0.
\end{array}\right .\label{mm}
\end{eqnarray}
The last $m \times m$ system (\ref{mm}) is 
\begin{eqnarray}
\left ( 
\begin{array}{cccc}
f_{n-m+1} & f_{n-m+2} & \cdots & f_{n}\\
f_{n-m+2}& f_{n-m+3} & \cdots & f_{n+1}\\
\vdots & & & \vdots\\
f_{n} & f_{n+1} & \ldots & f_{m+n-1}
\end{array}
\right )
\left (
\begin{array}{c}
q_m \\
q_{m-1}\\
\vdots\\
q_1
\end{array}
\right )
=
-q_0 \left (
\begin{array}{c}
f_{n+1} \\
f_{n+2}\\
\vdots\\
f_{m+n}
\end{array}
\right )
\label{sys}
\end{eqnarray}
and therefore has a unique solution because $f$ is a Stieltjes series
and $q_0=1$. 
The first system (\ref{nm}) is then solved
for $p_0, \ldots, p_n$. 
\pn 
\tiret
System (\ref{mm}) may be also written 
\begin{eqnarray*}
\left ( 
\begin{array}{cccc}
f_{n-m+2} & f_{n-m+3} & \cdots & f_{n+1}\\
f_{n-m+3}& f_{n-m+4} & \cdots & f_{n+2}\\
\vdots & & & \vdots\\
f_{n+1} & f_{n+2} & \ldots & f_{m+n}
\end{array}
\right )
\left (
\begin{array}{c}
q_{m-1} \\
q_{m-2}\\
\vdots\\
q_0
\end{array}
\right )
=
-q_m \left (
\begin{array}{c}
f_{n-m+1} \\
f_{n-m+2}\\
\vdots\\
f_{n}
\end{array}
\right ). 
\end{eqnarray*}
We thus deduce that if $q_m=0$ then $Q_m=0$ and $Q_m(0)=0$. 
\pn
\tiret
With the last equation of (\ref{nm}) 
and (\ref{mm}), we have 
\begin{eqnarray*}
\left ( 
\begin{array}{cccc}
f_{n-m} & f_{n-m+1} & \cdots & f_{n}\\
f_{n-m+1}& f_{n-m+2} & \cdots & f_{n+1}\\
\vdots & & & \vdots\\
f_{n} & f_{n+1} & \ldots & f_{m+n}
\end{array}
\right )
\left (
\begin{array}{c}
q_m \\
q_{m-1}\\
\vdots\\
q_0
\end{array}
\right )
=
\left (
\begin{array}{c}
p_n \\
0\\
\vdots\\
0
\end{array}
\right ) .
\end{eqnarray*}
We thus deduce that $p_n \not =0$. 
\qed%
\end{proof}
\begin{remark}
The system (\ref{sys}) shows that if $Q_m(0)=0$, then $Q_m=0$.
\end{remark}
\begin{definition}
We say that 
$f^{[n/m]} = P_n/Q_m$ is the Pad{\'e} approximant of order $(n,m)$ of $f$. 
\end{definition}
We will make use of a very useful theorem concerning Stieltjes
series.   
\begin{theorem} \label{bmt}
Let $f(x)$ be a Stieltjes series with radius of convergence $R$ and let
us denote by $f^{[n/m]}$ its Pad{\'e} approximant $P_n/Q_m$. Then 
\begin{enumerate}
\item \label{poles}
$Q_m$ has exactly $m$ real roots in $]R, + \infty [.$
\item Let $f^{[n/m]}(x) = \sum_{k \geq 1} f_k^{[n/m]} x^k$. We have
\begin{enumerate}
\item for $1 \leq k \leq n+m$, $0 < f_k^{[n/m]} = f_k$. \label{2a}
\item $0 \leq f_{n+m+1}^{[n/m]} < f_{n+m+1}$.\label{2b}  
\item for $k \geq n+m+1$, $0 \leq f_k^{[n/m]} \leq f_k$.\label{2c}  
\end{enumerate}
\end{enumerate}
\end{theorem} 
\begin{proof}
The assertion (\ref{poles}) is proved in \cite{BG}, p. 220. Note that
the authors use the function $f(-z)$. 
Assertion (\ref{2a}) is a consequence of the Pad{\'e} approximation
definition. Assertion (\ref{2c}) is proved in \cite{BG}, p. 212. Note
that the authors have shown that $0\leq f_k^{[n/m]} \leq f_k$. 
Suppose now that $f_{n+m+1}^{[n/m]} = f_{n+m+1}$. From theorem
{\bf \ref{pade-t}}, we would have $\deg P_n + \deg Q_m = n+m+1$ and this is
not the case. We thus have $f_{n+m+1}^{[n/m]} < f_{n+m+1}$.
\qed%
\end{proof}
We thus deduce
\begin{corollary}\label{cnm}
Let $m\leq n$. There are polynomials $P_n \in \RR_n[u]$, 
$Q_m \in \RR_m[u]$ and $F_{n,m} \in \RR[v]$  such that 
$$
Q_m(u)v-P_n(u) = v^{n+m+1} F_{n,m}(v),
$$
where $F_{n,m}(0)=1$.
Furthermore, we have $F_{n,m}(v)>0$ when $v \in [0,1]$, 
$\deg P_n=n$, $\deg Q_m = m$ and $Q_m(u)>0$ for $u \in [0,1]$. 
\end{corollary}
\begin{proof}
$\phi$ is a Stieltjes series and because 
$\Frac{\phi_{n+1}}{\phi_n} \equi_{n \to \infty} 1 -\Frac 3{2n}$, we
deduce that its radius of convergence $R$ is 1 and that 
$\sum_{n \geq 1} \phi_n = \phi(1) = 1$. 
Let $\phi^{[n/m]} = {P_n}/{Q_m}$ be the Pad{\'e} approximant of 
$\phi$, we deduce that 
$$
\phi(u) - \phi^{[n/m]}(u) = \sum_{k\geq n+m+1} (\phi_k - \phi_k^{[n/m]})u^k
= u^{n+m+1} \psi_{n,m}(u), \ 0 \le u \leq 1. 
$$ 
We have $\psi_{n,m}(u)>0$ for $u \in [0,1])$, from theorem
{\bf\ref{bmt}}, (\ref{2b}).  
From $Q_m(0)=1$ and theorem {\bf \ref{bmt}}, (\ref{poles}) we get
$Q_m(u)>0$ for $u \in [0,1]$, and 
$$
vQ_m(u)-P_n(u)=u^{n+m+1} \psi_{n,m}(u) Q_m(u) >0. 
$$
On the other hand, as $4 u = v(v-3)^2 \equi_{v\to 0} 9v$, we deduce that
$vQ_m(u)-P_n(u)$ is a polynomial in $v$ with $0$ as root of order 
$n+m+1$. We deduce that 
$$ 
vQ_m(u)-P_n(u) = v^{n+m+1} F_{n,m}(v),$$ 
where $F_{n,m}$ is a polynomial. 
\qed%
\end{proof}
We deduce 
\begin{proposition}\label{cn}
There exists a family $C_n$ in $\vect (W_0, \ldots, W_n)$, 
such that 
$$
C_n = t^{2n+1} F_n, \ F_n(0)= 1.
$$
Furthermore, $\deg C_n = 2n+2\pent n2+1$ and $F_n(t)>0$ for $t \in [-2,2]$. 
\end{proposition}
\begin{proof}
Let us consider 
$$
C_{k,l} (t) = v Q_l(u) - P_k(u)= v^{k+l+1}F_{k,l}(v) 
$$
given by corollary {\bf \ref{cnm}}. Note that $Q_l(u)>0$ for $u \in [0,1]$.
\pn
\tiret
If $t \in [-1,1]$, we have $u,v \in [0,1]$ and the announced result by
corollary {\bf \ref{cnm}}. 
\pn
\tiret
We have $u([1,2])= u([-2,-1])=u([0,1])=[0,1]$. Let 
$\abs t \in [1,2]$. There exists $t_1 \in ]0,1]$, such that 
$u=u(t)=u(t_1)=u_1$ and we have $v=v(t)=t^2\geq t_1^2=v(t_1)=v_1$. 
We deduce
\begin{eqnarray*}
C_{k,l}(t) = 
v Q_l(u) - P_k(u) &=& 
v Q_l(u_1) - P_k(u_1)
\\
&\geq& 
v_1 Q_l(u_1) - P_k(u_1)>0.
\end{eqnarray*}
In conclusion, for $t\in [-2,2]$, we have $F_{k,l}(t)>0$.
\pn
\tiret $Q_l(u) \in (T_2+2)\RR_l[T_6+2]$ and $P_k \in
\RR_k[T_6+2]$. We thus deduce that $C_{k,l} \in \RR[T_6] \oplus T_2 \RR[T_6]$.
Note that $\deg C_{k,l} = \max(6k+2,6l)$.  
\pn
\tiret
If $n=2k+1$, let $C_n = t\cdot C_{k,k}$. 
If $n=2k$, let $C_n = t\cdot C_{k,k-1}$. 
$C_n$ has degree $2n+2\pent n2+1$ and therefore $C_n \in
\vect(W_0,\ldots,W_n)$.   
\qed%
\end{proof}
\begin{remark}
We have proved the existence of $C_n$. This is an upper-triangular basis of $E$
with respect to the $W_i$. It is unique and it can be computed by
simple LU-decomposition of the matrix whose lines are the $W_i$.
\end{remark}
\section{Conclusion}
We have shown in this paper the existence of plane polynomial curves
of degree $(3,N+ 2\pent N4+1)$ having the required properties. We
think that they are of minimal lexicographic degrees (it is true for
$N=3,5,7,9$). This question is related to the following
question: where are the real zeros of polynomials in 
$\vect(V_k, k \not = 2 \mod 3)$? 
We guess that such polynomials cannot have too many zeroes in
$[-1,1]$. It would give a lower bound for the degrees
of the torus knots approximation by polynomial curves.  
\pn We have not given explicit formulas for our polynomials.
We have just shown that they can be found by solving some
explicit linear system. In a near future, we hope we will be able to
give explicit function of the degree $N$.

\end{document}